\documentclass[11pt]{amsart}
\usepackage[T1]{fontenc}
\usepackage[english]{babel}
\usepackage{kpfonts}

\usepackage{amsmath,amssymb,amsthm,enumerate,xspace}
\usepackage{comment}
\usepackage{accents}
\usepackage[colorlinks=true,citecolor=blue,linkcolor=blue]{hyperref}
\usepackage[all]{xy}
\usepackage{tikz}
\usepackage{mathrsfs}
\usepackage{cleveref}

\numberwithin{equation}{section}

\newtheorem{thm}{Theorem}

\newtheorem*{prob*}{Problem}

\newtheorem{prop}[thm]{Proposition}
\newtheorem{lem}[thm]{Lemma}
\newtheorem{cor}[thm]{Corollary}

\newtheorem*{iprob*}{Problem}

\theoremstyle{definition}

\newtheorem*{defi*}{Definition}
\newtheorem{rem}[thm]{Remark}

\newtheorem*{acks*}{Acknowledgements}

\newcommand{\sU}{\mathscr{U}}
\newcommand{\ZZ}{\mathbf{Z}}

\newcommand{\RR}{\mathbf{R}}
\newcommand{\NN}{\mathbf{N}}

\newcommand{\CC}{\mathbf{C}}
\newcommand{\QQ}{\mathbf{Q}}

\newcommand{\SL}{\mathbf{SL}}

\newcommand{\PSL}{\mathbf{PSL}}

\newcommand{\SSS}{\mathbf{S}}

\newcommand{\Sym}{{\rm Sym}}

\newcommand{\se}{\subseteq}
\newcommand{\inv}{^{-1}}

\newcommand{\lra}{\longrightarrow}

\newcommand{\wt}{\widetilde}
\newcommand{\wh}{\widehat}

\newcommand{\ol}{\overline}

\DeclareMathOperator{\Hom}{Hom}
\DeclareMathOperator{\Homeo}{Homeo}
\newcommand{\cont}{\mathfrak c}
\newcommand{\alo}{\aleph_0}
\title[Ulam's problem for nilpotent groups]{Lie groups as permutation groups:\\ Ulam's problem in the nilpotent case}
\date{October 2021}

\author[Nicolas Monod]{Nicolas Monod}
\address{EPFL, Switzerland}
\begin{document}

\begin{abstract}
Ulam asked whether every connected Lie group can be represented on a countable structure. This is known in the linear case. We establish it for the first family of non-linear groups, namely in the nilpotent case. Further context is discussed to illustrate the relevance of nilpotent groups for Ulam's problem.
\end{abstract}
\maketitle


\section{Introduction}

Cayley's principle states that every group is a permutation group. In particular, every finite group can be faithfully represented in a symmetric group $\Sym(n)$ for some $n\in \NN$.

For infinite groups, the situation becomes immediately more interesting. We shall investigate it at the light of the following notion:

\begin{defi*}
A group is \textbf{countably representable} if it admits a faithful action on some countable set.
\end{defi*}

In other words, $G$ is countably representable if it can be realised as a subgroup of $\Sym(\alo)$. This is also equivalent to admitting a faithful linear representation on some countable vector space.

Of course the question of countable representation only arises for groups whose cardinality does not exceed the continuum cardinal $\cont=2^{\alo}$ since this is the size of $\Sym(\alo)$. On the other hand the question is non-trivial only for uncountable groups.

This leaves us with a very rich landscape of groups to investigate, because the groups of size~$\cont$ include all (non-discrete) Polish groups, thus in particular all (non-discrete) locally compact second countable groups, and still more particularly all (non-trivial) connected Lie groups.

At first sight, the familiar Lie groups have no obvious countable representation; such a representation would be non-continuous, even non-measurable, and indeed in Solovay's model~\cite{Solovay} they do not have any. But in November 1935, Schreier and Ulam proved the following (see Section~\ref{sec:prelim} for the elementary argument):

\begin{prop}[Schreier--Ulam]\label{prop:R}
The group $\RR$ is countably representable.
\end{prop}

This is Item~95 in the translated Scottish Book~\cite{Scottish58}, although we only found it as a question in the original version~\cite{Scottish_halpha}.

Ulam later asked a question that remains open to this day:

\begin{prob*}[Ulam]
Is every connected Lie group countably representable?
\end{prob*}

\noindent
(See~\cite[II.7]{Ulam_book60} and~\cite[V.2]{Ulam_book64}.)

This question is really the key to the wider world of locally compact groups, because the solution to Hilbert's fifth problem leads to the following fact (see Section~\ref{sec:other}):

\begin{prop}\label{prop:Hilbert}
If the answer to Ulam's Problem is affirmative, then every locally compact second countable group is countably representable.
\end{prop}

On the ominous side, some of the most familiar Polish groups spectacularly fail to be countably representable. For instance, it follows from the Rosendal--Solecki automatic continuity theorem~\cite{Rosendal-Solecki} that any action of $G=\Homeo(\SSS^1)$ or $G=\Homeo(\RR)$ on a countable set is trivial on the neutral component of $G$ (which is of index two in this case). This was generalised to arbitrary compact manifolds by Mann~\cite{Mann16}. Another very familiar group without any non-trivial action on a countable set is the unitary group of the infinite-dimensional separable Hilbert space, by Tsankov's automatic continuity~\cite{Tsankov13}.

In the positive direction, \Cref{prop:R} can be generalised to show that in the abelian case there is no obstruction at all beyond the obvious size restriction (see Section~\ref{sec:other} for two proofs):

\begin{prop}[de Bruijn~\cite{deBruijn64}]\label{prop:abelian}
Every abelian group is countably representable as long as its cardinality does not exceed $\cont$.
\end{prop}

Returning to Lie groups and Ulam's Problem, there is a much stronger result than \Cref{prop:R}, namely:

\begin{thm}[Thomas~\cite{Thomas99}]
Every \emph{linear} Lie group is countably representable.
\end{thm}
\noindent
(This was first established in~\S2 of~\cite{Thomas99}; the proof was rediscovered by Kallman~\cite{Kallman00} and later Ershov--Churkin~\cite{Ershov-Churkin}; see Section~\ref{sec:nilpotent} for the argument.)

\medskip

Therefore it remains to settle Ulam's Problem for non-linear Lie groups. The first example is Birkhoff's reduced Heisenberg group~\cite{Birkhoff36}, a central extension of $\RR^2$ with center $\SSS^1$. More generally, a connected nilpotent Lie groups is non-linear unless it is a semi-direct product of a \emph{simply connected} group by a torus~\cite[16.2.7]{Hilgert-Neeb}. Thus nilpotent Lie groups include a large family of non-linear groups; we can answer Ulam's question in this case:


\begin{thm}\label{thm:nilpotent}
Every nilpotent connected Lie group is countably representable.
\end{thm}

It is fair to ask whether this result just follows formally from the case of abelian groups by taking successive extensions and invoking a general stability property of the class of countably representable groups. It turns out that there is no such general stability: McKenzie~\cite{McKenzie71} has constructed a nilpotent group $M$ which is not countably representable, but is a central extension
\[
0 \lra \ZZ \lra M \lra B \lra 0
\]
where $B$ is the free $\ZZ /2\ZZ$-module of rank~$\cont$. Of course this group is not a Lie group, but it is nilpotent and of size~$\cont$. This shows in particular that countable representability is not preserved by central extensions (see Section~\ref{sec:Churkin}). A close variant of this group was considered by Churkin~\cite{Churkin05}. As we shall see, these obstructions to represantibility still hold for the corresponding extensions with finite center $\ZZ/2\ZZ$ instead of $\ZZ$, and for certain extensions of the form
\[
0 \lra \RR \lra H \lra \RR \lra 0.
\]
In fact, Churkin conjectured that his group is not a subgroup of any separable (Hausdorff) topological group~\cite[Rem.~3]{Churkin05}. We confirm this:

\begin{thm}\label{thm:Churkin:top}
Any homomorphism from McKenzie's or Churkin's nilpotent group to any separable Hausdorff topological group is trivial on the center.

The same statement holds for homomorphisms to Lindel{\"o}f Hausdorff topological groups.
\end{thm}

\noindent
(This is a strengthening of not being countably representable since $\Sym(\alo)$ is Polish, in particular \emph{both} separable and Lindel{\"o}f.)

We note that Lindel{\"o}f groups include notably all compact groups and more generally any countable product of $\sigma$-compact groups, see~\cite[8.1(ii)]{ComfortHB}.

\bigskip

Besides nilpotent groups, perhaps the best known example of a non-linear Lie group is Cartan's example~\cite[\S V]{Cartan36} of the universal cover of $\SL_2(\RR)$, or already its double cover, the metaplectic group. 
There, the methods of the present note seem to fall short:

\begin{prob*}
Let $G=\SL_n(\RR)$ and let $\wt G$ be its universal cover. Is $\wt G$ countably representable?
\end{prob*}

\noindent
Recall that $\wt G$ is a central extension of $G$ by $\ZZ/2\ZZ$ when $n\geq 3$ and by $\ZZ$ when $n=2$. We observe below (\Cref{rem:SL}) that the approach used in the previous cases cannot work for these groups.

\begin{rem}
Ulam did not explicitly assume connectedness, but this is only a matter of terminology:

Under the convention that Lie groups are second countable, they have at most countably many connected components and hence we are immediately reduced to the connected case (\Cref{lem:induction} below).

If on the other hand no restriction is made, then we can trivially obtain disconnected Lie groups that are not countably representable by considering a suitable group as zero-dimensional Lie group.
\end{rem}

\begin{acks*}
Ulam's problem was first mentioned to me by Pierre de la Harpe. During earlier attempts for higher rank simple groups, I had the benefit and pleasure of conversations with Alex Lubotzky and with Matthew Morrow. I am indebted to Alexander Pinus for sending me the article~\cite{Churkin05}. I am grateful to Yves Cornulier for his comments on a first version; Yves alerted me to a gap in the earlier proof of Theorem~\ref{thm:Churkin:top} and provided me with additional historical references.
\end{acks*}


\section{Preliminary observations}\label{sec:prelim}

We call a subgroup $H<G$ \textbf{cocountable} if it is of countable index in $G$, i.e.\ if the coset space $G/H$ is countable.

Then $G$ is countably representable if and only if it admits a sequence of cocountable subgroups $G_n <G$ ($n\in \NN$) with trivial intersection $\bigcap_n G_n$. There is no loss of generality in assuming that this sequence is decreasing. We shall often use tacitly the fact that the intersection of a cocountable group with any subgroup remains cocountable in the latter.

\medskip

Since the (direct) product of any family of symmetric groups acts faithfully on the disjoint union of the corresponding sets, we have:

\begin{lem}\label{lem:product}
The class of countably representable groups is closed under finite and countable products.\qed
\end{lem}

The class being closed under passing to subgroups, this further implies:

\begin{lem}\label{lem:inverse}
The class of countably representable groups is closed under taking inverse limits of countable inverse systems.\qed
\end{lem}

The Schreier--Ulam proof that $\RR$ is countably representable was probably as follows.

\begin{proof}[Proof of \Cref{prop:R}]
By \Cref{lem:product}, the countable power $\QQ^\NN$ is countably representable. Since this group is uniquely divisible, it is a $\QQ$-vector space. Its dimension is $\cont$ and therefore it is isomorphic to $\RR$.
\end{proof}

We will establish countable representability for general abelian groups in Section~\ref{sec:other} but already record the following:

\begin{cor}\label{cor:circle}
The circle group $\SSS^1$ is countably representable.
\end{cor}

\begin{proof}
Viewing again $\RR$ as a $\QQ$-vector space, we split off $\QQ$ and obtain an isomorphism $\RR/\ZZ \cong  \QQ/\ZZ \oplus \RR$; the result now follows from \Cref{prop:R} (and \Cref{lem:product}).
\end{proof}

Another basic hereditary property relies on the technique of induced actions:

\begin{lem}\label{lem:induction}
If $G$ admits a countably representable cocountable subgroup, then $G$ is itself countably representable.
\end{lem}

\begin{proof}
Suppose that the group $G$ admits a cocountable subgroup $G_0<G$ which acts faithfully on a countable set $X$. Define $Y=(G\times X)/G_0$, where $G_0$ acts diagonally on $G\times X$, from the left on $G$. Then $Y$ is countable since it admits a bijection with $(G_0\backslash G) \times X$. We endow $Y$ with a $G$-action by letting $G$ act by right multiplication on $G$ and trivially on $Y$, noting that this action on $G\times X$ descends indeed to $Y$. The resulting action is faithful.
\end{proof}

\section{Proof of Theorem~\ref{thm:nilpotent}}\label{sec:nilpotent}

We begin by recalling Thomas's proof for linear Lie groups:

Let $K$ be the field of Puiseux series over $\ol \QQ$ endowed with its valuation $v\colon K \to \QQ \cup\{\infty\}$; let $V$ be its valuation ring. Then $\SL_n(V)$ is a cocountable subgroup of $\SL_n(K)$, see~\cite[Thm.~2.5]{Thomas99}. The argument for that is a refinement of the fact that \emph{discrete} valuations give rise to an action on the Bruhat--Tits building, which is countable: here $v$ is not discrete on $K$, but $K$ is the countable union of the subfields $\ol\QQ((t^{1/q}))$ of Laurent series on $t^{1/q}$, on each of which $v$ is discrete.

By the Newton--Puiseux theorem, $K$ is algebraically closed. Since moreover $K$ has characteristic zero and cardinality $\cont$, there is a field isomorphism $K\cong \CC$. We denote by $V_\CC$ the image in $\CC$ of the valuation ring of $K$ and by $L_\CC$ the corresponding maximal ideal. Define  the congruence subgroup $\SL_m(V_\CC; L_\CC)$ to be the kernel of the reduction map $\SL_m(V_\CC) \to  \SL_m(\ol \QQ)$. Then $\SL_m(V_\CC; L_\CC)$ is a cocountable subgroup of $\SL_m(\CC)$ since $\SL_m(\ol \QQ)$ is countable.

At this point it follows already that $\PSL_m(\CC)$ is countably representable (since it is simple); finally, any linear Lie group can be embedded into some $\PSL_m(\CC)$.

\medskip

Now we give a more explicit expansion of this argument in order to be able to adapt it for some non-linear groups.

Define the decreasing sequence of ideals $L_n < V_\CC$ by $\{x:  v(x) > n \}$ under the chosen isomorphism $K\cong \CC$. Thus the decreasing sequence of groups $\SL_m(V_\CC; L_n)$ has trivial intersection.

\begin{lem}\label{lem:congr}
The congruence subgroups $\SL_m(V_\CC; L_n)$ are cocountable in $\SL_m(\CC)$.
\end{lem}

\begin{proof}
It suffices to show that $\SL_m(V_\CC; L_n)$ is cocountable in $\SL_m(V_\CC)$, i.e.\ that the ring $V_\CC / L_n$ is countable. This is the case because any given element of this ring is represented by a \emph{finite} sum
\[
\sum_{j =0}^{q n} \lambda_j t^{j/q}
\]
for some choice of denominator $q\in \NN_{>0}$ and elements $\lambda_j\in\ol \QQ$ indexed by $j\in \NN$; here $t$ is the formal variable of Puiseux series.
\end{proof}

We will need the following additional observation:

\begin{lem}\label{lem:val}
Let $C<\CC$ be a finitely generated subgroup and let $z\in \CC$ be an element not in $C$. Then there is $n$ such that $L_n$ does not meet the coset $z+ C$.
\end{lem}

\begin{proof}
Viewing everything in $K$ rather than in $\CC$, the claim is that $v$ is uniformly bounded over $z+C$; note that this coset does not contain~$0$. Letting $D$ be the group generated by $C$ and $z$, it suffices to prove the following claim:

For every finitely generated subgroup $D<K$, the valuation $v$ is uniformly bounded over the non-zero elements of $D$.

The group $D$ is free abelian of finite rank $r\geq 1$. We shall prove by induction on $r$ that $v$ takes at most $r+1$ distinct values on  $D$. The base case $r=1$ holds because $v( p x) = v(x)$ for all $x\in K$ and all non-zero $p\in\ZZ$; this gives the two values $v(x)$ and $v(0)=\infty$ on $D$ in rank one.

We turn to the induction step for $r\geq 2$. Let $x_1, \ldots, x_r$ be a basis of $D$. We order the basis so that for some $1\leq s \leq r$ we have
\[
v(x_1) = \cdots = v(x_s) < v(x_j) \kern3mm\forall\, j>s.
\]
Given any element $x= p_1 x_1 + \cdots + p_r x_r$ in $D$ (with $p_i\in \ZZ$), write $x'= p_1 x_1 + \cdots + p_s x_s$ and consider its residue $\omega\in\ol \QQ$ at $v(x_1)$, i.e.\ the coefficient of $t^{v(x_1)}$ in $x'$. We have $v(x) = v(x') = v(x_1)$ unless $\omega$ vanishes. However, $\omega=0$ defines a subgroup $D_0$ of $D$; this subgroup is still free abelian but of smaller rank since $x\mapsto \omega$ defines a non-trivial torsion-free quotient of $D$. Therefore, the induction hypothesis limits at $r$ the number of distinct values on $D_0$, which completes the induction step once we add the value $v(x_1)$.
\end{proof}

(In fact we will only use the case where $C$ is cyclic, but this is still rank $r=2$ in the above proof.)

\begin{proof}[End of proof of Theorem~\ref{thm:nilpotent}]
Let $G$ be a nilpotent connected Lie group. We refer to~\cite[\S11.2]{Hilgert-Neeb} (especially Thm.~11.2.10 therein) for background on the structure of $G$. We recall in particular the following. The fundamental group of $G$ is a free abelian group of finite rank $d\geq 0$. Moreover, denoting by $\wt G$ the universal cover of $G$, there is a discrete central subgroup $\Gamma<\wt G$ isomorphic to $\pi_1(G)$ with $\wt G \to \wt G/\Gamma \cong G$ implementing the covering morphism. In addition, $\Gamma$ is a lattice in a central subgroup $Z \cong \RR^d$ of $\wt G$, so that $T=Z/\Gamma$ is a central $d$-torus in $G$.

Our proof is by induction on $d$. The base case $d=0$ holds by Thomas's result since $G$ is then simply connected and hence linear~\cite[16.2.7]{Hilgert-Neeb}.

The case $d=1$ is the main case because the general case is easily reduced to it; hence we treat it separately. Since $Z$ is connected, the quotient $H= G/T \cong \wt G/Z$ is simply connected. Therefore, by the case $d=0$, there is a countable family of cocountable subgroups $H_n<H$ such that every non-trivial element of $H$ is outside some $H_n$. Taking pre-images gives cocountable subgroups $G_n<G$ such that every element of $G$ \emph{not in T} lies outside $G_n$ for some $n$. It therefore suffices to produce another sequence of cocountable subgroups $K_n<G$ such that every non-trivial element of $T$ lies outside some $K_n$.

Since $\wt G$ is simply connected, it is linear; more precisely, by the Lie--Kolchin theorem, we can identify it with a Lie subgroup of the upper unitriangular subgroup $U_m$ of $\SL_m(\RR)$ for some $m$. We define $\wh K_n< \wt G$ by taking the intersection of $\wt G$ with the congruence subgroup $\SL_m(V_\CC; L_n)$ under such an embedding; thus $\wh K_n$ is cocountable in $\wt G$ by \Cref{lem:congr}. Let $K_n<G$ be the image of $\wh K_n$ in $G$, which is cocountable.

Let us reformulate for $\wh K_n < \wt G$ the desired property of $K_n$: we need to show that every element $\zeta\in Z< \wt G$ which is not in $\Gamma$ lies outside the group $\Gamma \wh K_n$ for some $n$. Indeed, $\Gamma \wh K_n$ is the pre-image of $K_n$ in $\wt G$. Equivalently, we need to prove that $\zeta\Gamma \cap \wh K_n =  \varnothing$.

The lower central series $U_m^i$ of $U_m$ is defined by $U_m^0=U_m$ and $U_m^{i+1} = [U_m^i, U_m]$. Concretely, the $(p,q)$-entry of a matrix in $U_m^i$ vanishes when $p<q \leq p+i$. Consider the largest $i$ such that $\zeta\in U_m^i$; since $\zeta$ is non-trivial, $i\leq m-2$. Moreover, the evaluation map $\epsilon\colon U_m \to \RR$ at a $(p, p+i+1)$-entry is a group homomorphism on  $U_m^i$, namely one of the canonical coordinates of $U_m^i/ U_m^{i+1} \cong \RR^{m-i-1}$. By maximality of $i$, we can choose $p$ such that $\epsilon(\zeta)$ is non-zero. Recalling that $Z$ has dimension $d=1$, it follows that $\epsilon$ restricts to an isomorphism $\epsilon|_Z\colon Z \to \RR$. In conclusion, $z=\epsilon(\zeta)$ lies outside $C=\epsilon(\Gamma)$. By \Cref{lem:val}, there is $n$ such that $z+C$ does not meet $L_n$. This implies that $\zeta\Gamma$ does not meet $\wh K_n$, as was to be shown.

Finally, we perform the induction step for $d\geq 2$. We choose two distinct elements $\gamma_1, \gamma_2$ of some basis of $\Gamma$, write $\Gamma_i=\langle \gamma_i \rangle$ for the group generated by $\gamma_i$ and $Z_i< Z$ for the one-parameter subgroup containing $\gamma_i$. The quotient $T_i = Z_i/ \Gamma_i$ is a central circle group in $G$ and thus we can form $H_i= G/T_i$. Then the fundamental group of $H_i$ is isomorphic to $\Gamma/ \Gamma_i$, e.g.\ because $\wt G /Z_i$ is a universal cover of $H_i$ since $Z_i$ is connected.

Thus the fundamental group of $H_i$ has rank $d-1$. By the induction hypothesis, there are two countable families $H_{i,n}$ of cocountable subgroups of $H_1$ respectively $H_2$ such that any non-trivial element of $H_i$ lies outside $H_{i,n}$ for some $n$. Taking pre-images, we obtain cocountable subgroups $G_{i,n}< G$ such that every $g\notin T_i$ lies outside some $G_{i,n}$. Since $T_1 \cap T_2$ is trivial, this completes the induction step.
\end{proof}

\begin{rem}\label{rem:SL}
Consider the question whether the double cover $\wt G$ of $G=\SL_3(\RR)$ is countably representable. Since $G$ is countably representable, it all amounts to finding a cocountable subgroup of $\wt G$ which meets trivially the kernel of the covering map. Equivalently, the question is:

\itshape
Does $G$ itself contain a cocountable subgroup that can be lifted to $\wt G$?\upshape

This cannot be proved by using as above a cocountable congruence subgroup $G_n = G\cap \SL_3(V_\CC;L_n)$. Indeed we claim that $G_n$ does not lift to $\wt G$.

The reason is as follows. The set $1+L_n$ meets $\RR$ densely because $L_n \cap \RR$ is a cocountable subgroup of $\RR$. Therefore, $1+L_n$ contains a negative real number $x$. It follows that $G_n$ contains the diagonal matrices with entries $(x, x\inv, 1)$ and $(x, 1,  x\inv)$ (noting that $1+L_n$ is a multiplicative subgroup of $V_\CC^\times$). The classical relation between Steinberg symbols and central extensions of $G$, as explained e.g.\ on p.~95 of Milnor's book~\cite{MilnorK}, shows that these two diagonal matrices do not admit commuting lifts in $\wt G$. This is because the Steinberg symbol for $\wt G$ is the Hilbert quadratic residue symbol (see p.~104 in~\cite{MilnorK}), but $x<0$ means that the Hilbert symbol is non-trivial. This excludes a minori any lift of the group $G_n$.
\end{rem}

\section{Auxiliary proofs}\label{sec:other}

The reduction from general locally compact second countable groups to connected Lie groups is standard:

\begin{proof}[Proof of Proposition~\ref{prop:Hilbert}]
Let $G$ be a locally compact second countable group. Denoting the neutral component of $G$ by $G_0$, the quotient $G/G_0$ is locally compact and totally disconnected; therefore, van Dantzig's theorem~\cite[p.~18]{vanDantzigPhD} implies that $G/G_0$ contains a compact-open subgroup $K$. We denote by $G_1<G$ the pre-image of $K$ in $G$; thus $G_1/G_0$ is compact.

Since $G$ is second countable, we can choose a countable base $\sU$ of neighbourhoods $U\subseteq G_1$ of the identity in $G_1$. By the solution of Hilbert's fifth problem (see e.g.~\cite[4.6]{Montgomery-Zippin}), there is a compact normal subgroup $N_U\lhd G_1$ contained in $U$ such that $G_1/N_U$ is a Lie group. Note that $G_1/N_U$ has finitely many connected components; therefore, since we assumed that the answer to Ulam's Problem is affirmative, $G_1/N_U$ is countably representable.

On the other hand, $G_1$ is the inverse limit of the system  $(G_1/N_U)_{U\in \sU}$. Since $\sU$ is countable, Lemma~\ref{lem:inverse} implies that $G_1$ is countably representable. Finally, we use again that $G$ is second countable to deduce that $G/G_1$ is countable since $G_1$ is open in $G$. Applying Lemma~\ref{lem:induction}, we conclude that $G$ itself is countably representable.
\end{proof}

Next, we propose two proofs that every abelian group is countably representable, unless its cardinality is too large. Yves Cornulier informed me that the second proof below was already known to de~Bruijn, see Thm.~4.3 in~\cite{deBruijn64}.

\begin{proof}[Analytic proof of Proposition~\ref{prop:abelian}]
Let $G$ be an abelian group with $|G|\leq \cont$. We consider $G$ as a discrete group, so that its Pontryagin dual $\widehat G=\Hom(G, \SSS^1)$ is a compact group. Then the topological weight of $\widehat G$, namely the smallest cardinality of a base for its topology, is also~$\leq \cont$, see~\cite[7.76]{Hofmann-Morris}. It now follows from Engelking's theorem~\cite[Thm.~10]{Engelking65} that $\widehat G$ is separable; the fact that Engelking's theorem applies to compact groups is due to the fact that they are \emph{dyadic spaces}, see e.g.~\cite[10.40]{Hofmann-Morris}. We can therefore choose a sequence of characters $f_n\colon G\to \SSS^1$ that is dense in $\widehat G$.

Since $\SSS^1$ is countably representable by \Cref{cor:circle}, there is a sequence of morphisms $h_m$ from $\SSS^1$ to countable quotients $h_m(\SSS^1)$ such that the intersection of the kernels of all $h_m$ is trivial. It now suffices to verify that for every non-trivial $g\in G$ there is a pair $(n,m)$ such that $h_m(f_n(g))$ is non-trivial.

To this end, Pontryagin duality ensures that there is a character $f\in \widehat G$ such that $f(g)$ is non-trivial. By density of $(f_n)$, we can fix $n$ such that $f_n(g)$ is non-trivial in $\SSS^1$, and it remains only to choose $h_m$ in terms of $f_n(g)$.
\end{proof}

\begin{proof}[Algebraic proof of Proposition~\ref{prop:abelian}]
Let $G$ be an abelian group with $|G|\leq \cont$. The injective hull $D$ of $G$ contains $G$, and its construction shows that $D$ still has size~$\leq \cont$ (see e.g.\ the proof of Thm.~17 in~\cite[III]{Griffith_book}). It therefore suffices to consider divisible abelian groups such as $D$. By the structure theory of divisible abelian groups (see e.g.~\cite[Thm.~23.1]{Fuchs_bookI}), $D$ is a sum of a $\QQ$-vector space and of $p$-torsion parts, where $p$ ranges over all primes. It suffices to consider each piece individually because of \Cref{lem:product}. The $\QQ$-vector space is handled by the argument of \Cref{prop:R} since it has dimension at most~$\cont$.

The $p$-torsion part is of the form $\bigoplus_\kappa \ZZ(p^\infty)$ for some cardinal $\kappa$, where $\ZZ(p^\infty)$ is the Pr\"ufer $p$-group (see again~\cite[Thm.~23.1]{Fuchs_bookI}); upon enlarging it we can assume $\kappa=\cont=2^{\alo}$. It is therefore contained in $\ZZ(p^\infty)^{\alo}$, see Ex.~3 in~\cite[\S23]{Fuchs_bookI}. We conclude again by \Cref{lem:product} since $\ZZ(p^\infty)$ is countable.
\end{proof}

\section{On McKenzie's and Churkin's examples}\label{sec:Churkin}

Churkin~\cite{Churkin05} defines a group $G$ by generators and relations as follows. Let $I$ be an index set of cardinality~$\cont$.
\[
G \ = \ \Big\langle  a_i, b_i, c \ (i\in I) : \text{all generators commute, except $[a_i, b_i] = c \  (\forall \, i)$} \Big\rangle
\]
Earlier, McKenzie defined an almost identical group $M$ with the same notation as above but imposing in addition that all $a_i$ and $b_i$ have order two (proof of Thm.~4 in~\cite{McKenzie71}). We shall work with $G$, but the arguments work equally well for $M$.

Observe that there is an epimorphism from $G$ to a free abelian group $A$ of continuum rank~$\cont$. Indeed define $A$ to be free on a set $\{\bar a_i, \bar b_i : i\in I\}$ and map $a_i\mapsto \bar a_i$, $b_i\mapsto \bar b_i$. The kernel of this epimorphism is generated by $c$. In particular, $G$ is nilpotent and of size~$\cont$. (In the case of $M$, replace $A$ by the free $\ZZ/2\ZZ$-module $B$ on $\{\bar a_i, \bar b_i : i\in I\}$.)

\begin{thm}[McKenzie, Churkin]\label{thm:Churkin}
Every homomorphism from $M$ or $G$ to $\Sym(\alo)$ factors through $A$, respectively $B$.
\end{thm}

Since we found library access to~\cite{Churkin05} difficult, we reproduce its beautiful argument:

\begin{proof}[Churkin's proof of \Cref{thm:Churkin}]
Let $X$ be a countable set with a $G$-action. We need to show that $c$ fixes any given $x\in X$. Consider the map $I\to X$, $i\mapsto a_i x$. Since $I$ is uncountable, it contains an uncountable subset $J$ such that $a_i x$ is constant over $i\in J$. For the same reason, there is an uncountable subset $K\se J$ such that $b_i x$ is constant over $i\in K$. Choose three distinct indices $i,j,k$ all in $K$. Since both $a_j\inv a_i$ and $b_k\inv b_i$ fix $x$, so does $[a_j\inv a_i, b_k\inv b_i]$. On the other hand, this commutator is $c$ because $a_j$ commutes to $a_i, b_k, b_i$ and $b_k$ commutes to $a_j, a_i, b_i$. The same proof works for $M$.
\end{proof}

To deduce that $G$ is not countably representable, it remains of course to prove that the subgroup generated by $c$ in $G$ is not trivial. Unfortunately the argument given in~\cite{Churkin05}, aiming to map $G$ onto the integral Heisenberg group, does not work --- and indeed \Cref{thm:Churkin} rules out witnessing the non-triviality of $c$ in any countable quotient.

There is nonetheless a direct argument. Consider the central extension
\[
0 \lra \ZZ \lra E \lra A \lra 0
\]
given by the following two-cocycle (factor set) $f\colon A \times A \to \ZZ$. Denote by $\alpha_i\colon A\to \ZZ$ the coefficient map of $\bar a_i$ and likewise $\beta_i$ for $\bar b_i$. Then define $f(x,y)= \sum_{i\in I} \alpha_i(x) \beta_i(y)$, noting that the sum is finite for each $x,y\in A$. This is a normalised two-cocycle and hence defines an extension $E$ as above (see e.g.\ Ex.~IV.3.8 and~V.6.5 in~\cite{Brown_coho} for the connection between $\ZZ$-bilinear maps and two-cocycles on abelian groups). Explicitly, recall that the underlying set for $E$ is $\ZZ \times A$ with multiplication
\[
(n, x)\cdot (m,y) = \big(n+m+f(x,y), x+y\big).
\]
The presentation of $G$ implies that there is an epimorphism $G\to E$ mapping $c$ to $(1_\ZZ,0_A)$ and compatible with the epimorphisms $G\to A$ and $E\to A$. The latter property ensures that this map $G\to E$ is actually an isomorphism $G\cong E$. By construction, $(1_\ZZ,0_A)$ is non-trivial in $E$ and indeed has infinite order.

This confirms that $G$ is not countably representable and further implies that for all $n\geq 2$ the quotient of $G$ by $\langle c^n\rangle$ gives an extension of $A$ by $\ZZ/n \ZZ$ which is not countably representable either.

\medskip

The above explicit construction can also be used as follows. The cocycle $f$ extends by bilinearity to a $\QQ$-valued cocycle on the free $\QQ$-module on $\{\bar a_i, \bar b_i : i\in I\}$, which is isomorphic to $\RR$. We therefore obtain an extension
\[
0 \lra \QQ \lra F \lra \RR \lra 0
\]
containing Churkin's group. Defining $H=F\times \RR$, this further implies the following curiosity, in marked contrast to our \Cref{thm:nilpotent}:

\begin{cor}
There exists a nilpotent group $H$ given by a central extension
\[
0 \lra \RR \lra H \lra \RR \lra 0
\]
which is not countably representable.\qed
\end{cor}

Finally, we present the proof of \Cref{thm:Churkin:top}, establishing that $G$ and $M$ do not admit any faithful representation to any Hausdorff topological group that is separable or Lindel{\"o}f. This contains \Cref{thm:Churkin} since $\Sym(\alo)$ is a Polish group.

\begin{proof}[Proof of \Cref{thm:Churkin:top}]
We keep the notation introduced above for Churkin's group $G$ and for its explicit construction as central extension $E$ (the proof for $M$ is the same).

We claim that if $G$ is covered by countably many left translates of some subset $W\se G$, then the central generator $c$ is a commutator of two elements of the product set $W\inv W$.

Indeed there is an uncountable subset $J\se I$ such that all $a_j$ with $j\in J$ belong to a fixed translate of $W$, say $r W$ for some $r\in G$. Likewise there is $K\se J$ infinite and $s\in G$ such that $b_k\in s W$ for all $k\in K$. As in \Cref{thm:Churkin}, we observe that $c=[a_j\inv a_i, b_k\inv b_i]$ whenever $i,j,k$ are distinct. Taking all three indices in $K$, both $a_j\inv a_i$ and $b_k\inv b_i$ belong to $W\inv W$, confirming the claim.

Consider now a group homomorphism $\pi\colon G \to H$ to some separable or Lindel{\"o}f topological group $H$. It suffices to show that $\pi(c)$ belongs to every neighbourhood $U\se H$ of the identity in $H$. 

To that end, choose a symmetric neighbourhood $V$ of the identity in $H$ such that $V^{16} \se U$. There exists a sequence $\{h_n\}$ in $H$ such that $\{h_n V\}$ covers $H$. Indeed, in the Lindel{\"o}f case this holds by definition of that property. In the separable case, we can take any sequence $\{h_n\}$ that is dense in $H$ since then for any $h\in H$ the sequence $h_n\inv h$ must meet $V$.

In particular, $\pi(G)$ is covered by $\{h_n V\}$ and upon extracting we assume that moreover $h_n V$ meets $\pi(G)$ for every $n$. We then choose $g_n\in G$ with $\pi(g_n) \in h_n V$. It now follows
\[
h_n V \se \pi(g_n) V\inv V = \pi(g_n) V^2.
\]
Thus $\pi(G)$ is covered by the sequence $\pi(g_n) V^2$ and we conclude that $G$ is covered by countably many left translates of $\pi\inv(V^2)$.

We now apply the initial claim to $W=\pi\inv(V^2)$ and deduce that $c$ is a commutator of two elements of $W\inv W =\pi\inv(V^4)$. Therefore, $\pi(c)$ is the product of four elements of $V^4$ and thus belongs indeed to $V^{16} \se U$, completing the proof.
\end{proof}

The reader will have noticed that this proof is inspired by two sources: Churkin's proof above and the use of the Steinhaus property made in the separable case by Rosendal and Solecki in Prop.~2 of~\cite{Rosendal-Solecki}.


\bibliographystyle{halpha-abbrv} 
\bibliography{../BIB/ma_bib}

\def\cprime{$'$}
\begin{thebibliography}{Scottish}
\expandafter\ifx\csname url\endcsname\relax
  \def\url#1{\texttt{#1}}\fi
\expandafter\ifx\csname doi\endcsname\relax
  \def\doi#1{\burlalt{doi:#1}{http://dx.doi.org/#1}}\fi
\expandafter\ifx\csname urlprefix\endcsname\relax\def\urlprefix{URL }\fi
\expandafter\ifx\csname href\endcsname\relax
  \def\href#1#2{#2}\fi
\expandafter\ifx\csname burlalt\endcsname\relax
  \def\burlalt#1#2{\href{#2}{#1}}\fi

\bibitem[Bir36]{Birkhoff36}
G.~Birkhoff.
\newblock Lie groups simply isomorphic with no linear group.
\newblock {\em Bull. Amer. Math. Soc.}, 42(12):883--888, 1936.

\bibitem[Bro82]{Brown_coho}
K.~S. Brown.
\newblock {\em Cohomology of groups}, volume~87 of {\em Graduate Texts in
  Mathematics}.
\newblock Springer-Verlag, New York-Berlin, 1982.

\bibitem[Car36]{Cartan36}
{\'E}.~Cartan.
\newblock {La topologie des groupes de Lie. (Expos\'es de g\'eom\'etrie VIII)}.
\newblock {Actual. sci. industr. 358, 28 p.}, 1936.

\bibitem[Chu05]{Churkin05}
V.~A. Churkin.
\newblock Examples of groups with continuum cardinality that are not embeddable
  into the permutation group of a countable set (in {R}ussian).
\newblock In {\em Algebra and model theory 5}, pages 39--43. Novosibirsk State
  Tech. Univ., Novosibirsk, 2005.

\bibitem[Com84]{ComfortHB}
W.~W. Comfort.
\newblock Topological groups.
\newblock In {\em Handbook of set-theoretic topology}, pages 1143--1263.
  North-Holland, Amsterdam, 1984.

\bibitem[dB64]{deBruijn64}
N.~G. de~Bruijn.
\newblock Addendum to ``{E}mbedding theorems for infinite groups''.
\newblock {\em Nederl. Akad. Wetensch. Proc. Ser. A 67=Indag. Math.},
  26:594--595, 1964.

\bibitem[EC04]{Ershov-Churkin}
Y.~L. Ershov and V.~A. Churkin.
\newblock One problem of {U}lam.
\newblock {\em Dokl. Math.}, 70(3):896--898, 2004.

\bibitem[Eng65]{Engelking65}
R.~Engelking.
\newblock Cartesian products and dyadic spaces.
\newblock {\em Fund. Math.}, 57:287--304, 1965.

\bibitem[Fuc70]{Fuchs_bookI}
L.~Fuchs.
\newblock {\em Infinite abelian groups. {V}ol. {I}}.
\newblock Pure and Applied Mathematics, Vol. 36. Academic Press, New
  York-London, 1970.

\bibitem[Gri70]{Griffith_book}
P.~A. Griffith.
\newblock {\em Infinite abelian group theory}.
\newblock The University of Chicago Press, Chicago-London, 1970.

\bibitem[HM06]{Hofmann-Morris}
K.~H. Hofmann and S.~A. Morris.
\newblock {\em The structure of compact groups}, volume~25 of {\em de Gruyter
  Studies in Mathematics}.
\newblock Walter de Gruyter \& Co., Berlin, augmented edition, 2006.
\newblock A primer for the student---a handbook for the expert.

\bibitem[HN12]{Hilgert-Neeb}
J.~Hilgert and K.-H. Neeb.
\newblock {\em Structure and geometry of {L}ie groups}.
\newblock Springer Monographs in Mathematics. Springer, New York, 2012.

\bibitem[Kal00]{Kallman00}
R.~R. Kallman.
\newblock Every reasonably sized matrix group is a subgroup of {$S_\infty$}.
\newblock {\em Fund. Math.}, 164(1):35--40, 2000.

\bibitem[Man16]{Mann16}
K.~Mann.
\newblock Automatic continuity for homeomorphism groups and applications.
\newblock {\em Geom. Topol.}, 20(5):3033--3056, 2016.
\newblock With an appendix by Fr\'{e}d\'{e}ric Le Roux and Mann.

\bibitem[McK71]{McKenzie71}
R.~N.~W. McKenzie.
\newblock A note on subgroups of infinite symmetric groups.
\newblock {\em Nederl. Akad. Wetensch. Proc. Ser. A {\bf 74} = Indag. Math.},
  33:53--58, 1971.

\bibitem[Mil71]{MilnorK}
J.~Milnor.
\newblock {\em Introduction to algebraic {$K$}-theory}.
\newblock Princeton University Press, 1971.
\newblock Annals of Mathematics Studies, No. 72.

\bibitem[MZ55]{Montgomery-Zippin}
D.~Montgomery and L.~Zippin.
\newblock {\em Topological transformation groups}.
\newblock Interscience Publishers, New York-London, 1955.

\bibitem[RS07]{Rosendal-Solecki}
C.~Rosendal and S.~Solecki.
\newblock Automatic continuity of homomorphisms and fixed points on metric
  compacta.
\newblock {\em Israel J. Math.}, 162:349--371, 2007.

\bibitem[Scottish]{Scottish_halpha}
{K}si{\k{e}}ga {S}zkocka ({T}he {S}cottish {B}ook).
\newblock {N}otebook of the {S}cottish {C}af{\'e} in {L}w{\'o}w, 1935--1941.

\bibitem[Sol70]{Solovay}
R.~M. Solovay.
\newblock A model of set-theory in which every set of reals is {L}ebesgue
  measurable.
\newblock {\em Ann. of Math. (2)}, 92:1--56, 1970.

\bibitem[Tho99]{Thomas99}
S.~Thomas.
\newblock Infinite products of finite simple groups. {II}.
\newblock {\em J. Group Theory}, 2(4):401--434, 1999.

\bibitem[Tsa13]{Tsankov13}
T.~Tsankov.
\newblock Automatic continuity for the unitary group.
\newblock {\em Proc. Amer. Math. Soc.}, 141(10):3673--3680, 2013.

\bibitem[Ula58]{Scottish58}
S.~M. Ulam.
\newblock Untitled.
\newblock Typed manuscript sent to Edward Thomas Copson, based on the Scottish
  Book, 1958.

\bibitem[Ula60]{Ulam_book60}
S.~M. Ulam.
\newblock {\em A collection of mathematical problems}.
\newblock Interscience Tracts in Pure and Applied Mathematics, no. 8.
  Interscience Publishers, New York-London, 1960.

\bibitem[Ula64]{Ulam_book64}
S.~M. Ulam.
\newblock {\em Problems in modern mathematics}.
\newblock Science Editions John Wiley \& Sons, Inc., New York, 1964.

\bibitem[vD31]{vanDantzigPhD}
D.~van Dantzig.
\newblock {\em Studien over topologische algebra}.
\newblock PhD thesis, Groningen, 1931.

\end{thebibliography}

\end{document}